\documentclass[12pt]{article}
\usepackage{pifont}
\usepackage{amssymb,float}
\usepackage{amsthm,graphicx,pifont}
\usepackage{graphicx,psfrag,float,subfigure}
\usepackage{caption,enumerate}
\usepackage[scale=0.8,a4paper]{geometry}
\usepackage[colorlinks]{hyperref}
\usepackage{amsthm,amsmath,amssymb}
\usepackage{mathrsfs}
\newtheorem{theorem}{Theorem}[section]
\newtheorem{definition}[theorem]{Definition}
\newtheorem{lemma}[theorem]{Lemma}

\title{The spectral radii and extremal graphs of two types of minimal graphs}
\author{Liwen Lian, {Jinfeng Liu, Mengyuan Niu, Xiumei Wang\thanks{Corresponding author: Xiumei Wang. e-mail: wangxiumei@zzu.edu.cn.}}\\
{\small  School of Mathematics and Statistics, Zhengzhou University}\\
{\small Zhengzhou 450001, China}\\
}
\date{}\makeatother
\begin{document}
\maketitle
\begin{center}
\vskip 0.6cm
\end{center}

\begin{abstract}
A connected nontrivial graph $G$ is {\it matching covered} if every edge of $G$ is contained in some perfect matching of $G$. A matching covered graph $G$ is {\it minimal} if $G-e$ is not matching covered for each edge $e$ of $G$. A graph is said to be {\it factor-critical} if $G-v$ has a perfect matching for every $v\in V(G)$. A factor-critical graph $G$ is said to be {\it minimal factor-critical} if $G-e$ is not factor-critical graph for each edge $e\in E(G)$. In this paper, by employing ear decomposition and edge-exchange techniques, the greatest spectral radii of minimal matching covered bipartite graphs and minimal factor-critical graphs are determined, and the corresponding extremal graphs are characterized.
\end{abstract}


{\bf 2000 MR Subject Classification} \ 05C50, 05C70, 05C75

\section{Introduction}

Graphs considered in this paper are connected and simple. For standard graph-theoretical notation and terminology, we refer the readers to \cite{ref100} and \cite{ref17}. Let $G$ be a graph with vertex set $V(G)$ and edge set $E(G)$. The {\it order} of $G$ is denoted by $n$, i.e. $n=|V(G)|$. We write $K_{n}$ for the {\it complete graph} of order $n$. The {\it complement} $\overline{G}$ of $G$ is the simple graph whose vertex set is $V(G)$ and whose edges are the pairs of nonadjacent vertices of $G$. The {\it union} of graphs $G$ and $H$ is the graph $G\cup H$ with vertex set $V(G)\cup V(H)$ and edge set $E(G)\cup E(H)$. If $G$ and $H$ are edge-disjoint, $G\cup H$ can be denoted by $G+H$. The {\it join} $G\vee H$~is the graph obtained from $G\cup H$ by adding edges joining every vertex of $G$ to every vertex of $H$. We write $P_{n}$ and $C_{n}$ for the path and cycle of order $n$, respectively.

 A {\it matching} of $G$ is a set of pairwise nonadjacent edges. A {\it perfect matching} of $G$ is a matching which covers all vertices of $G$. A connected nontrivial graph $G$ is {\it matching covered} if every edge of $G$ is contained in some perfect matching of $G$. A matching covered graph $G$ is {\it minimal} if $G-e$ is not matching covered for each edge $e$ of $G$. Lov\'{a}sz and Plummer~\cite{ref16} gave an upper bound on the number of edges of a minimal matching covered bipartite graph and characterized extremal graphs. Zhang et al.~\cite{ref8} presented a complete characterization of minimal matching covered graphs which are claw-free. He et al.~\cite{ref9} proved that the minimum 
degree of a minimal matching covered graph other than $K_{2}$ is either 2 or 3. Mallik et al.~\cite{ref10} characterized minimal matching covered bipartite graphs with the number of vertices of degree 2 is $2(m-n+2)$.
 
A graph is said to be {\it factor-critical} if $G-v$ has a perfect matching for every $v\in V(G)$. A factor-critical graph $G$ is said to be {\it minimal} if $G-e$ is not factor-critical graph for each edge $e\in E(G)$. Favaron~\cite{ref11} and Yu~\cite{ref12} independently defined $k$-factor-critical graphs for any positive integer $k$. A graph is said to be {\it$k$-factor-critical} if the removal of any set of $k$ vertices results in a graph with a perfect matching. Zhang and Fan~\cite{ref13} established lower bounds on the number of edges of $G$ and the spectral radius of $G$, respectively, to guarantee that $G$ is $k$-factor-critical.  

Let $A(G)$ be the {\it adjacency matrix} of a graph $G$ and $\rho(G)$ be its {\it spectral radius}, which is the largest eigenvalue of $A(G)$. By Perron-Frobenius Theorem~\cite{ref32}, every connected graph $G$ of order $n$ has a positive unit eigenvector $\mathbf{x}=(x_{1}, x_{2}, \ldots, x_{n})^{T}$ corresponding to $\rho(G)$, which is called the {\it Perron vector} of $G$.
 
One of the most well-known problems on spectral graph theory is the following Brualdi-Solheid problem~\cite{ref26}.

\indent\noindent\textbf{Problem 1 (Brualdi-Solheid problem):} Given a set $\mathcal{G}$ of graphs, find a tight upper bound for the spectral radius of graphs in $\mathcal{G}$ and characterize the extremal graphs.

The problem is well studied in the literature for many classes of graphs, such as graphs with cut vertices~\cite{ref2}, cut edges~\cite{ref7}, given chromatic number~\cite{ref3}, diameter~\cite{ref5}, domination number~\cite{ref6}, and independence number~\cite{ref28, ref27}. Feng et al.~\cite{ref4} characterized the extremal graphs with maximum spectral radius among simple graphs of order $n$ and matching number $\beta$. Zhang et al.~\cite{ref30,ref31} characterized these graphs among $t$-connected graphs with matching number at most $\frac{n-k}{2}$ and among connected graphs with minimum degree $\delta$ and matching number at most $\frac{n-k}{2}$, where $t,k$ are two positive integers. In this paper, by employing ear decomposition and edge-exchange techniques, the greatest spectral radius of minimal matching covered bipartite graphs and minimal factor-critical graphs are determined, and the extremal graphs are characterized. The main results are as follows.
\begin{theorem}\label{t1} 
Let~$G$~be a minimal matching covered bipartite graph of order $n$. If~$n\leq4$, then $\rho(G)\leq2$, the equality holds if and only if~$G\cong C_{4}$$\mathrm{;}$ If~$n\geq 6$, then $\rho(G)\leq\frac{1+\sqrt{2n-3}}{2}$, the equality holds if and only if~$G\cong P^{*}_{3}$, where $P^{*}_{3}$~is the union of~$\frac{n-2}{2}$~paths of length of 3 with the same ends (see Figure 1).
\end{theorem}
\begin{figure}[htbp]
  \centering
  \includegraphics[width=0.3\textwidth]{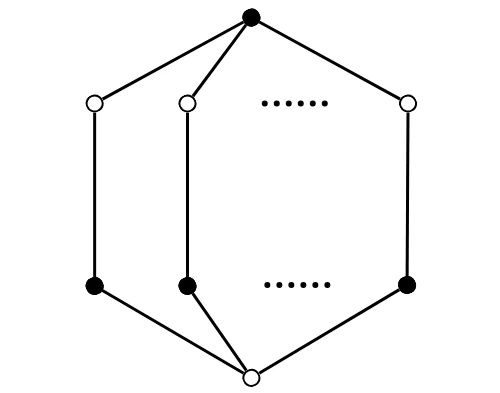}
  \caption{The graph $P^{*}_{3}$.}
\end{figure}

\begin{theorem}\label{t1}
Let~$G$~be a minimal factor-critical graph of order $n$. Then~$\rho(G)\leq \rho(K_{1}\vee \frac{n-1}{2}K_{2})$, the equality holds if and only if~$G\cong K_{1}\vee \frac{n-1}{2}K_{2}$.
\end{theorem}

\section{Preliminaries}
We begin with some definitions, which are used in the proof of the main results. For any subset $S$ of $V(G)$, let$~N_{G}(S)$ denote all vertices in $~V(G)\setminus S~$which are adjacent to at least one vertex in $S$. We denote by $G-S$ the subgraph of $G$ obtained from $G$ by deleting the vertices in $S$ together with their incident edges, and we denote by $G[S]$ the subgraph of $G$ induced by $S$. For a subset $E'$ of $E(G)$, $G-E'$ is the graph obtained from $G$ by deleting edges in $E'$. Similarly, $G+E'$ represents the graph obtained from $G$ by adding edges in $E'$. A path with ends $u$ and $v$~is called a ($u, v$)-{\it path}. For a path $P$ and two vertices $u$, $v$ on $P$, we denote by $P[u,v]$ the segment of $P$ with ends $u$ and $v$.

A {\it chord} of a cycle $C$ in a graph $G$ is an edge in $E(G)\setminus E(C)$ both of whose ends lie on $C$. A subgraph $H$ of $G$ is said to be {\it nice} if $G-V(H)$ has a perfect matching. In particular, if $H$ is a cycle, then $H$ is called a {\it nice cycle} of $G$. An {\it ear} of $G$ is an odd path whose ends lie in $G$ but whose internal vertices do not. An ear is {\it trivial} if it has only one edge, and {\it nontrivial} otherwise. {\it A family of parallel ears} is a set of ears that share the same ends and whose internal vertices pairwise distinct. The following gives the definition of the ear decomposition.
\begin{definition}
An ear decomposition of $G$ is a sequence of subgraphs $(G_{0}, G_{1}, G_{2}, \ldots, G_{k})$ of $G$ that satisfies the following conditions:
\begin{enumerate}[(i)]
\setlength{\itemsep}{-1ex}
\item $G_{0}$~is a subgraph of $G$,
\item $G_{i+1}=G_{i}+P_{i+1}$, where~$P_{i+1}$~is an ear of~$G_{i}$~$(0\leq i\leq k-1)$,
\item $G_{k}=G$.
\end{enumerate}

If $G$ has an ear decomposition $(G_{0}, G_{1}, G_{2}, \ldots, G_{k})$, we write $G=G_{0}+P_{1}+\cdots+P_{k}$, $G_{i+1}=G_{i}+P_{i+1} (0\leq i \leq k-1)$, and call $G_{0}+P_{1}+\cdots+P_{k}$ an ear decomposition of $G$. In particular, if $G_0$ is an odd cycle, it is called an odd ear decomposition of $G$. If $G$ is a bipartite graph, $G_0=K_2$, and $P_{i+1}$ connects vertices in different parts of $G_{i}$, then it is called a bipartite ear decomposition of $G$. At this time, $G_{1}=G_{0}+P_{1}$ is an even cycle.       
\end{definition}
\begin{definition}
Let $G$ be a matching covered bipartite graph with a bipartite ear decomposition $G=C+P_{1}+\cdots+P_{k}+\mathcal {P}_{k+1}+\mathcal{P}_{k+2}$, where $\mathcal{P}_{k+i}$ ($i=1,2$) is a family of parallel ears (it may contain only one ear), and at least one end of $\mathcal{P}_{k+i}$ lies on the ear $P_{k}$. If the ends of~$\mathcal {P}_{k+1}$ and $\mathcal {P}_{k+2}$ on $P_{k}$ appear in order (as shown in Figure 2), then we call $\mathcal {P}_{k+1}$ and $\mathcal{P}_{k+2}$ are compatible.          
\end{definition}
\begin{figure}[htbp]
  \centering
  \includegraphics[width=0.6\textwidth]{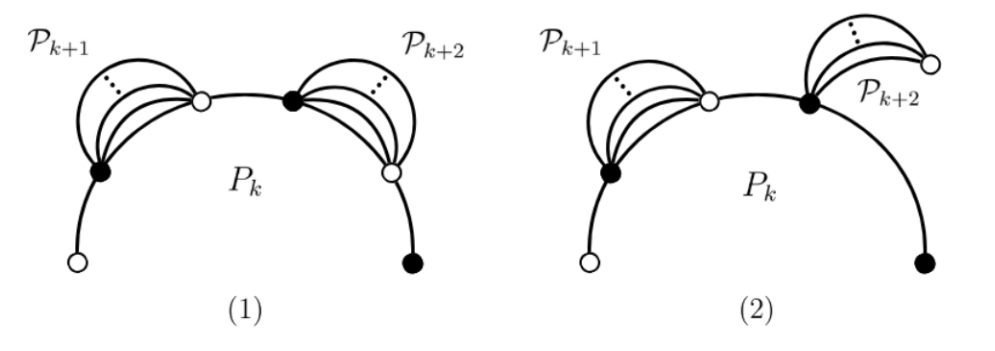}
  \caption{$P_{k}$, $\mathcal {P}_{k+1}$~and~$\mathcal {P}_{k+2}$}
  \vspace{-3mm}
\end{figure}
Next, we introduce some lemmas essential to the proof.
\begin{lemma}[\cite{ref16}]\label{m2}
A bipartite graph is matching covered if and only if it has a bipartite ear decomposition. 
\end{lemma}

\begin{lemma}[\cite{ref16}]\label{mg}
Any matching covered nice subgraph of a minimal matching covered graph is minimal, this is equivalent to saying that if a minimal matching covered bipartite graph $G$ has a bipartite ear decomposition~$C+P_{1}+P_{2}+\cdots+P_{k}$, then~$G_{i}=C+P_{1}+P_{2}+\cdots+P_{i}$~$(1\leq i\leq k)$ is also minimal matching covered.
\end{lemma}

\begin{lemma}[\cite{ref16}]\label{m4}
If $G$ is a minimal matching covered bipartite graph and is not a 4-cycle, then it contains no 4-cycle. 
\end{lemma}

\begin{lemma}[\cite{ref16}]\label{m5}
Let G be a matching covered bipartite graph. Then G is minimal if and only if no nice cycle has a chord.
\end{lemma}
 
Let$~A=(a_{ij})$ and $B=(b_{ij})$ be two $n\times n$ matrices. Define $A\leq B$ if $a_{ij}\leq b_{ij}$ for all $i$ and $j$, and define $A<B$ if $A\leq B$ and $A\neq B$.

\begin{lemma}[\cite{ref18}]\label{m9}
Let~$A=(a_{ij})$\,and\,$B=(b_{ij})$~be two~$n\times n$~matrices with the spectral radii $\rho(A)$ and $\rho(B)$, respectively. If $0\leq A\leq B$, then $\rho(A)\leq \rho(B)$. Furthermore, if $B$ is irreducible and~$0\leq A < B$, then~$\rho(A)< \rho(B)$.
\end{lemma}

It is known that the adjacency matrix of a connected graph is irreducible. Let~$E_{G}(T,S)$ denote the set of edges in graph~$G$~with one end in~$T$~and the other end in~$S$. 
\begin{lemma}[\cite{ref200}]\label{m11}
Let~$G$~be a connected graph of order$~n$, and $\mathbf{x}$ be the Perron vector of~$G$. Let~$S_{1}, S_{2}$ and $T$~be three nonempty disjoint subsets of $V(G)$, and~$E_{G}(T,S_{1})=\emptyset$, $|E_{G}(T,S_{2})|=|T||S_{2}|.$~If~$\sum\limits_{u\in S_{1}}x_{u}\geq \sum\limits_{v\in S_{2}}x_{v}$, then~$\rho(G+E_{\overline{G}}(T,S_{1})-E_{G}(T,S_{2}))>\rho(G).$  
\end{lemma}

An {\it automorphism} of a graph is an isomorphism of the graph to itself. We say that $u$ and $v$ are {\it equivalent} in $G$, if there exists an automorphism $p: G\rightarrow G$ such that $p(u)=v$. Vertex equivalence implies the following property of eigenvectors to $\rho(G)$.
\begin{lemma}[\cite{ref35}]\label{m10}
  Let $G$ be a connected graph of order $n$ and let $\mathbf{x}$ be the Perron vector of $A(G)$. If $u$ and $v$ are equivalent vertices in $G$, then $x_{u}=x_{v}$.
\end{lemma}
\begin{lemma}\label{mb}
The spectral radius of $P^{*}_{3}$ with order $n$ (see Figure 1) is the largest root of the equation~$\rho^{2}-\rho-\frac{n-2}{2}=0$.
\end{lemma}
\begin{proof}~Let $\rho=\rho(P^{*}_{3})$ and $\textbf{x}$~be the~Perron~vector of~$P^{*}_{3}$, denote the two vertices with degree~$\frac{n-2}{2}$~in~$ P^{*}_{3}$ by $u_{1}$~and $u_{2}$, and denote the vertices with degree 2 by~$u_{3}, u_{4}, \ldots, u_n$. For any vertex~$u_i$ in $ P^{*}_{3}$, we can obtain that$$\rho x_{u_{i}}=(A(P^{*}_{3})\textbf{x})_{u_i}=\sum\limits_{{u_{j}\in N_{P^{*}_{3}}(u_{i})}}x_{u_{j}}, 1\leq i\leq n.$$By Lemma~\ref{m10}, we have $x_{u_{1}}=x_{u_{2}}$, $x_{u_{3}}=\cdots=x_{u_{n}}$, and
$$\rho x_{u_{1}}=\Big(\frac{n-2}{2}\Big)x_{u_{3}}, ~\rho x_{u_{3}}=x_{u_{1}}+ x_{u_{3}}.$$
Combining the above two equations, we have $\rho^{2}-\rho-\frac{n-2}{2}=0$. The conclusion is proved.
\end{proof}
\begin{lemma}\label{m6}
Let~$G$~be a minimal matching covered bipartite graph with a bipartite ear decomposition~$C+P_{1}+P_{2}+\cdots+P_{k}$. Let~$G'=G+P_{k+1}$, where~$P_{k+1}$~is a nontrivial ear of~$G$. If the two ends of $P_{k+1}$ are two nonadjacent vertices and in different parts of $P_{k}$, then~$G'$~is a minimal matching covered bipartite graph.
\end{lemma}
\begin{proof}
According to Lemma~\ref{m2}, $G'$~is a matching covered bipartite graph. Let $C'$ be any nice cycle of $G'$, and $M'$ be a perfect matching in $G'-V(C')$. In order to prove that $G'$ is a minimal matching covered bipartite graph, by Lemma~\ref{m5}, it suffices to prove that $C'$ has no chords. Set $M=M'\cap E(G)$. Suppose that the two ends of $P_{k+1}$~are $u$ and $v$.

If $E(C')\cap E(P_{k+1})=\emptyset $, then $C'$ is in $G$. If~$M'\cap E(P_{k+1})$~does not cover $u$ and $v$, then $M$ is a perfect matching of $G-V(C')$, and $C'$ is also a nice cycle of $G$. If~$M'\cap E(P_{k+1})$~covers $u$ and $v$, then $P_{k}[u,v]$ is not contained in $C'$. So~$M \triangle E(P_{k}[u,v])$~is a perfect matching of~$G-V(C')$, and $C'$ is also a nice cycle of $G$. Since $G$ is a minimal matching covered bipartite graph, $C'$ has no chords by Lemma~\ref{m5}.

If $E(C')\cap E(P_{k+1})\neq \emptyset$, then $P_{k+1}$~is contained in $C'$. Note that the degree of the internal vertices of $P_{k}$~is 2 in $G$. If the cycle $C'$ is the union of the path~$P_{k}[u,v]$~and the ear $P_{k+1}$, then $C'$ contains no chords. Otherwise, if $C'$ has a chord, then this chord is $uv$ that is in $G$, and $u$ and $v$ are the ends of $P_{k}$. If $uv$ is one of the ears $P_{1},P_{2},\ldots,P_{k-1}$, then $G$ has a trivial ear, which contradicts the fact that $G$ is a minimal matching covered bipartite graph. If $uv$ is on one of the ears $P_{1},P_{2},\ldots,P_{k-1}$, say $P_{i}$, then the ear decomposition of $G$ can be adjusted to $C+P_{1}+\cdots+P_{i-1}+(P_{i}-uv+P_{k})+P_{i+1}+\cdots+P_{k-1}+uv$. At this time, $uv$~is a trivial ear of $G$, which contradicts that $G$ is a minimal matching covered bipartite graph.

If~$E(C') \cap E(P_{k}[u,v])=\emptyset$, replace the path$~P_{k+1}$~in $C'$ with the path $P_{k}[u,v]$, and denote the resulting cycle by $C''$. Then\,$M\setminus(M'\cap E(P_{k}[u,v]))$\,is a perfect matching of $G-V(C'')$. Therefore, $C''$ is a nice cycle of $G$. Since $G$ is a minimal matching covered bipartite graph, the graph $C''$ has no chords by Lemma \ref{m5}. Thus $C'$ has no chords in $G'$.

To sum up, the nice cycles of $G'$ have no chords. The proof is complete.
\end{proof}

\begin{lemma}\label{mx}
Let $G$ be a minimal matching covered bipartite graph with a bipartite ear decomposition~$C+P_{1}+\cdots+P_{k}+P_{k+1}+\cdots+P_{k+r}$, where $P_{k+i}$ has at least one end on $P_{k}$, $i=1,2,\ldots,r$. If $G'=G+P_{k+r+1}$, where $P_{k+r+1}$ is a nontrivial ear of $G$, its two ends are two nonadjacent vertices in different parts of $P_{k}$, and $P_{k+r+1}$ is compatible with $P_{k+i}$~$(1\leq i\leq r)$, then $G'$ is a minimal matching covered bipartite graph.     
\end{lemma}
\begin{proof}

By Lemma \ref{m2}, $G'$ is a matching covered bipartite graph. According to Lemma \ref{m5}, to prove that $G'$ is a minimal matching covered bipartite graph, it suffices to show that any nice cycle of $G'$ has no chords. Let $C'$ be a nice cycle of $G'$. Now we prove that $C'$ has no chords. Since $C'$ is a nice cycle of $G'$, $G'-V(C')$ has a perfect matching $M'$. Let $M=M'\cap E(G)$. We denote the two ends of $P_{k+r+1}$ by $u$ and $v$.

If$~E(C')\cap E(P_{k+r+1})=\emptyset$, then $C'$ is in $G$. If $M'\cap E(P_{k+r+1})$ does not cover $u$ and $v$, then $M$ is a perfect matching of $G-V(C')$, and thus $C'$ is a nice cycle of $G$. If $M'\cap E(P_{k+r+1})$~covers $u$ and $v$. Since $P_{k+r+1}$ is compatible with $P_{k+i}$~($1\leq i \leq r$), the internal vertices of $P_{k}[u,v]$ all have degree 2 in $G'$. Hence, $M\triangle E(P_{k}[u,v])$ is a perfect matching of $G-V(C')$, and thus $C'$ is a nice cycle of $G$. Since $G$ is a minimal matching covered bipartite graph, the cycle $C'$ has no chords by Lemma \ref{m5}.  

If $E(C')\cap E(P_{k+r+1})\neq \emptyset$, then $E(P_{k+r+1})\subseteq E(C')$. If $C'$ is the union of the path $P_{k}[u,v]$ and $P_{k+r+1}$, then $C'$ has no chords. If $E(C')\cap E(P_{k}[u,v])=\emptyset$, replace $P_{k+r+1}$ in $C'$ with $P_{k}[u,v]$, and denote the resulting cycle by $C''$. Let $M_{1}$ be a perfect matching of $P_{k}$, then $M\setminus(M_{1}\cap E(P_{k}[u,v]))$ is a perfect matching of $G-V(C'')$. Therefore, $C''$ is a nice cycle of $G$. Since $G$ is a minimal matching covered bipartite graph, the cycle $C''$ has no chords by Lemma \ref{m5}. Therefore, the nice cycle $C'$ also has no chords in $G'$.  

To sum up, any nice cycle in $G'$ has no chords. Therefore, $G'$ is a minimal matching covered bipartite graph. The proof is complete. 
\end{proof}

\section{Minimal matching covered bipartite graphs}

In this section, we give the upper bound of the spectral radius of minimal matching covered bipartite graphs and characterize the extremal graphs. Let $G$ be a minimal matching covered bipartite graph with a bipartite ear decomposition $C+P_{1}+P_{2}+\cdots+P_{k}$. We define the grade of ears in this ear decomposition. The ear $C$ is called a {\it 0-grade ear}. The ear $P_{1}$ is called a {\it 1-grade ear}. For ear $P_i$~$(i\geq 2)$, if one end of $P_i$ is an internal vertex of a $k$-grade ear and the other end is an internal vertex of a $l$-grade ear, where $k\leq l$, then $P_i$ is called a {\it $(l+1)$-grade ear}. The {\it grade number} of the graph $G$ is the highest level of all ears among all possible ear decompositions of $G$. Given an ear decomposition of the graph $G$, denote the set of all $i$-grade ears in $G$ by $\mathcal {P}^i$. Then the ear decomposition of $G$ with grade number $m$ can be expressed as~$G=C+ \mathcal {P}^{1}+ \mathcal {P}^{2}+\cdots+ \mathcal {P}^{m}$. Let $P^{m-1}$ be an ear in $\mathcal{P}^{m-1}$. If $P$ is an $m$-grade ear and has an end that is the internal vertex of $P^{m-1}$, then $P$ is called an $m$-grade ear on $P^{m-1}$. For the convenience of subsequent discussion, we introduce the following notation: let $x$ and $y$ be two vertices of $P^{m-1}$, and the set of all $m$-grade ears with both ends on $P^{m-1}[x,y]$ is denoted by $\mathcal{P}^{m}[x,y]$. 

In the sequel of this section, suppose that $G$ is a graph with the greatest spectral radius among all minimal matching covered bipartite graphs, and the grade number of $G$ is $m$.  Let $G^{*}$ be the graph obtained from $G$ by removing all the internal vertices of the $m$-grade ears on $P^{m-1}$. Let $G^{*}_{1}$ be the graph obtained from $G^{*}$ by removing the internal vertices of the ear $P^ {m-1}$. By Lemma \ref{mg}, $G^{*}$ and $G^{*}_{1}$ are minimal matching covered bipartite graphs. First, we study the compatibility of $m$-grade ears on $P^{m-1}$ and give the following lemma.        
\begin{lemma}\label{m12}
Let $u$ and $u'$ be two ends of $P^{m-1}$, $\mathcal{P}^{m}_{1}$ and $\mathcal{P}^{m}_{2}$ be any two families of parallel $m$-grade ears on $P^{m-1}$. Denote the ends of ears in $\mathcal{P}^{m}_{1}$ by $v$, $v'$, and the ends of ears in $\mathcal{P}^{m}_{2}$ by $w$, $w'$.
\begin{enumerate}[(i)]
\setlength{\itemsep}{-1ex}
\item If $|\{v,v'\} \cap \{w,w'\}|\neq\emptyset$, then $\{v,v'\}=\{w,w'\}$$\mathrm{;}$ 
\item If $|\{v,v'\} \cap \{w,w'\}|=\emptyset$, then the ears in $\mathcal{P}^{m}_{1}$ and $\mathcal{P}^{m}_{2}$ are compatible.  
\end{enumerate}
\end{lemma}  

\begin{proof}
Let $\mathcal{P}^{m}_{1}=\{P^{m}_{11}, P^{m}_{12},\ldots, P^{m}_{1l} \},~\mathcal{P}^{m}_{2}=\{P^{m}_{21}, P^{m}_{22},\ldots, P^{m}_{2l'}\},~l\geq1,~l'\geq1$. Assume, without loss of generality, that $v$ and $w$ are internal vertices of $P^{m-1}$ ($v',w'$~may or may not be on $P^{m-1}$).       

 $(i)$~Suppose that~$|\{v, v'\}\cap \{w, w'\}|\neq\emptyset$. Assume, without loss of generality, that $v=w$, $v'\neq w'$, and~$x_{v'}\geq x_{w'}$. Denote the neighbors of $w'$ on ears of $\mathcal{P}^{m}_{2}$ by $w_{1}',w_{2}',\ldots,w_{l'}'$. Let $G'=G-\sum\limits_{i=1}\limits^{l'}w'w_{i}'+\sum\limits_{i=1}\limits^{l'}v'w_{i}'$. By Lemma \ref{m2}, $G'$ is a matching covered bipartite graph. Remove the internal vertices of ears of $\mathcal{P}^{m}_{2}$ from $G$ to obtain a graph $G''$. By Lemma \ref{mg}, $G''$ is a minimal matching covered bipartite graph. The graph $G'$ is the union of $G''$ and parallel ears $\cup_{i=1}^{l'}(P^{m}_{2i}[w,w_{i}']+w_{i}'v')$. Then $G'$ is a minimal matching covered bipartite graph by Lemma \ref{m6}. By Lemma \ref{m11}, $\rho(G')> \rho(G)$, which contradicts the choice of $G$. 

$(ii)$~Suppose that $|\{v,v'\} \cap \{w,w'\}|=\emptyset$. Then the ends of ears in $\mathcal{P}^{m}_{1}$ and $\mathcal{P}^{m}_{2}$ are different. Recall that $u$ and $u'$ are two ends of $P^{m-1}$. Now we prove ears in $\mathcal{P}^{m}_{1}$ and $\mathcal{P}^{m}_{2}$ are compatible by contradiction. Suppose, without loss of generality, that at least one of the ends of ears in $\mathcal{P}^{m}_{2}$ is an internal vertex of $P^{m-1}[v, v']$, where $v$ is closer to $u$ than $v'$. Then, adjust $P^{m-1}+\mathcal{P}^{m}_{1}+\mathcal{P}^{m}_{2}$ in the ear decomposition of $G$ to $$(P^{m-1}[u,v]+P^{m}_{11}+P^{m-1}[v',u'])+\sum\limits_{i=2}\limits^{l}P^{m}_{1i}+P^{m-1}[v,v']+\mathcal{P}^{m}_{2},$$while keeping the other ears unchanged. We obtain a new ear decomposition of $G$. In this new ear decomposition, $P^{m-1}[u,v]+P^{m}_{11}+P^{m-1}[v',u']$ is an $(m-1)$-grade ear, $P^{m-1}[v,v']$ and $P^{m}_{1i}$ ($2\leq i\leq l$) are $m$-grade ears, $\mathcal{P}^{m}_{2}$ is a set of $(m+1)$-grade ears. This contradicts the assumption that the grade number of $G$ is $m$. The proof is complete.
\end{proof} 
Next, based on the distribution of the ends of ears on $P^{m-1}$, where $m\geq2$, we characterize minimal matching covered bipartite graphs with a greater spectral radius than $G$.

\begin{lemma}\label{a}
If the ends of all $m$-grade ears on $P^{m-1}$ are on $P^{m-1}$, then there exists a minimal matching covered bipartite graph $G'$ satisfying $\rho(G')>\rho(G)$. 
\end{lemma}
\begin{proof}
Suppose $P^{m}_{1}, P^{m}_{2}, \ldots, P^{m}_{l}$ are ears of a family of $m$-grade parallel ears with ends $v$ and $v'$ on $P^{m-1}$. It is possible that $l=1$. Suppose, without loss of generality, among the ends of all $m$-grade ears on $P^{m-1}$, the distance between $v'$~and\,$u'$ is the shortest. Denote the neighbors of $v$ on the ears $P^ {m}_{1}, P^{m}_{2}, \ldots, P^{m}_{l}$ by $v_{1}, v_{2}, \ldots, v_{l}$, respectively, and the neighbors of $v'$ on the ears $P^ {m}_{1}, P^{m}_{2}, \ldots, P^{m}_{l}$ by $v_{1}', v_{2}', \ldots, v_{l}'$, respectively. Based on whether $u'$ is equal to $v'$, we discuss the following two cases.

\indent \textbf{Case\,1.} $u'=v'$. Since $P^{m}_{1}, P^{m}_{2}, \ldots, P^{m}_{l}$ are $m$-grade ears, we have $u\neq v.$

If $x_{u}\geq x_{v}$, let $G'=G-\sum\limits_{i=1}\limits^{l}vv_{i}+\sum\limits_{i=1}\limits^{l}uv_{i}$. By Lemma \ref{m11}, $\rho(G')>\rho(G)$. Now we prove that $G'$ is a minimal matching covered bipartite graph. Let $G_{1}'$ be the graph obtained from $G'$ by removing the internal vertices of ears in $\mathcal {P}^{m}[u,v]$. By Lemma \ref{m6}, $G_{1}'$ is a minimal matching covered bipartite graph. By Lemma \ref{m12}, the families of $m$-grade parallel ears in $\mathcal{P}^{m}[u,v]$ are pairwise compatible. Hence by Lemmas \ref{m6} and \ref{mx}, $G'=G_{1}'+\mathcal {P}^{m}[u,v]$ is a minimal matching covered bipartite graph.             

If $x_{u}<x_{v}$, let $\mathcal{U}$ be the set of neighbors of $u$ on $P^{m-1}$ and on the ears in $\mathcal {P}^{m}[u,u']$, one end of which is $u$, and let $$G'=G-\sum\limits_{w\in N_{G}(u)\backslash \mathcal{U}}uw+\sum\limits_{w\in N_{G}(u)\backslash \mathcal{U}}wv+uu'.$$ Then $\rho(G')> \rho(G)$ by Lemmas \ref{m9} and \ref{m11}. In $G'$, remove the internal vertices of ears in $\mathcal {P}^{m}[u,v]$, the internal vertices of $P^ {m}_{i}$ for $1\leq i\leq l$, the internal vertices of $P^{m-1}[u',v]$, and the internal vertices of $P^{m-1}[v,u]+uu'$, denote the resulting graph by $G_{1}'$. It is obvious that $G^{*}_{1}\cong G_{1}'$. Recall that $G^{*}_{1}$ is a minimal matching covered bipartite graph, so is $G_{1}'$. By Lemma \ref{mg}, $G^{*}_{1}+P^ {m-1}$ is a minimal matching covered bipartite graph. Since $G$ is a minimal matching covered bipartite graph and the length of $P^{m-1}[u',v]$ and $P^{m-1}[v,u]+uu'$ are at least 3, $G_{2}'=G_{1}'+P^{m-1}[u',v]+(P^{m-1}[v,u]+uu')+\sum\limits_{i=1}\limits^{l}P^{m}_{i}$ is a minimal matching covered  bipartite graph by Lemma \ref{m6}. By Lemma \ref{m12}, the families of $m$-grade parallel ears in $\mathcal{P}^{m}[u,v]$ are pairwise compatible. Combining Lemma \ref{m6} with Lemma \ref{mx}, we have $G'=G_{2}'+\mathcal {P}^{m}[u,v]$ is a minimal matching covered bipartite graph.                             

\indent \textbf{Case\,2.}~$u'\neq v'$. If $u=v$, it can be reduced to Case 1. It suffices to consider the case where $u\neq v$. At this time, both $v$ and $v'$ are internal vertices of $P^{m-1}$. Let $\mathcal{U}$ be the set of neighbors of $u$ on the ear $P^{m-1}$ and on the $m$-grade ears on $P^{m-1}$ with $u$ as an end. Let $u''$ be the neighbor of $u'$ on $P^{m-1}$ ($u''$ may be the same vertex as $v'$). We shall discuss the following four subcases based on the relative values of $x_{u}$ and $x_{v}$, as well as $x_{u'}$ and $x_{v'}$.       

If $x_{u}>x_{v}$ and $x_{u'}>x_{v'}$, let$$G'=G-\sum\limits_{i=1}\limits^{l}vv_{i}+\sum\limits_{i=1}\limits^{l}uv_{i}-\sum\limits_{i=1}\limits^{l}v'v_{i}'
+\sum\limits_{i=1}\limits^{l}u'v_{i}'.$$Then $\rho(G')>\rho(G)$ by Lemma \ref{m11}. Recall that $G^{*}$ is a minimal matching covered bipartite graph. In $G'$, remove the internal vertices of ears in $\mathcal {P}^{m}[u,v]$, the internal vertices of $u'v_{i}'+P_{i}^{m}[v_{i}',v_{i}]+v_{i}u$ for $2\leq i\leq l$, and the internal vertices of $P^{m-1}$, denote the resulting graph by $G_{1}'$. Then $G_{1}'$ is the graph obtained from $G^{*}$ by replacing the ear $P^{m-1}$ with $u'v_{1}'+P_{1}^{m}[v_{1}',v_{1}]+v_{1}u$. Hence, $G_{1}'$ is a minimal matching covered bipartite graph. Thus $G_{2}'=G_{1}'+P^{m-1}+\sum\limits_{i=2}\limits^{l}(u'v_{i}'+P^{m}_{i}[v_{i}',v_{i}]+v_{i}u)$ is a minimal matching covered bipartite graph by Lemma \ref{m6}. Lemma \ref{m12} implies that the $m$-grade parallel ears in $\mathcal{P}^{m}[u,v]$ are pairwise compatible. By Lemmas~\ref{m6}~and~\ref{mx}, $G'=G_{2}'+\mathcal {P}^{m}[u,v]$ is a minimal matching covered bipartite graph. 

If $x_{u}> x_{v}$ and $x_{u'}\leq x_{v'}$, let $v''$ be the neighbor of $v$ on $P^{m-1}[v,v']$, and $$G'=G-\sum\limits_{w\in N_{G}(u')\backslash \{u''\}}u'w+\sum\limits_{w\in N_{G}(u')\backslash \{u''\}}v'w-\sum\limits_{i=1}\limits^{l}vv_{i}
+\sum\limits_{i=1}\limits^{l}uv_{i}-vv''+uv''+vu'.$$ Then $\rho(G')> \rho(G)$ by Lemmas \ref{m9} and \ref{m11}. Next, we prove that $G'$ is a minimal matching covered bipartite graph. In $G'$, remove the internal vertices of the ears in $\mathcal {P}^{m}[u,v]$, the internal vertices of $P^{m-1}[v',v'']+v''u$, the internal vertices of $P^{m-1}[u,v]+vu'+P^{m-1}[u',v']$, and the internal vertices of $P_{i}^{m}[v',v_{i}]+v_{i}u$ for $1\leq i\leq l$, denote the resulting graph by $G_{1}'$. Clearly, $G_{1}'\cong G_{1}^{*}$, where the vertices $u$ and $u'$ in $G_{1}^{*}$ correspond to the vertices $u$ and $v'$ in $G_{1}'$, respectively. Recall that $G_{1}^{*}$ is a minimal matching covered bipartite graph, so is $G_{1}'$. Since $G$ is a minimal matching covered bipartite graph, the length of $P^{m-1}[v,v']$ is at least 3. Thus the length of $P^{m-1}[v',v'']+v''u$ is at least 3. Recall that $G^{*}=G_{1}^{*}+P^{m-1}$ is a minimal matching covered bipartite graph. Therefore,  $G_{2}'=G_{1}'+(P^{m-1}[v',v'']+v''u)$ is a minimal matching covered bipartite graph. It follows from Lemma \ref{m6} that$$G_{3}'=G_{2}'+(P^{m-1}[u,v]+vu'+P^{m-1}[u',v'])+\sum\limits_{i=1}\limits^{l}(P^{m}_{i}[v',v_{i}]+v_{i}u)$$ is a minimal matching covered bipartite graph. According to Lemma \ref{m6}, the $m$-grade parallel ears in $\mathcal{P}^{m}[u,v]$ are pairwise compatible. Combining Lemma \ref{m6} with Lemma \ref{mx}, we have $G'=G_{3}'+\mathcal {P}^{m}[u,v]$~is a minimal matching covered bipartite graph.                                   

If $x_ {u}\leq x_{v}$ and $x_{u'}> x_{v'}$, exchange the labels of $u$ and $u'$, as well as $v$ and $v'$. The proof is the same as the case $x_{u}> x_{v}$ and $x_{u'}\leq x_{v'}$. 

If $x_ {u}\leq x_{v}$ and $x_{u'}\leq x_{v'}$, let $$G'=G-\sum\limits_{w\in N_{G}(u)\backslash \mathcal{U}}uw+\sum\limits_{w\in N_{G}(u)\backslash \mathcal{U}}vw-\sum\limits_{s\in N_{G}(u') \backslash \{u''\}}u's+\sum\limits_{s\in N_{G}(u')\backslash \{u''\}}v's+uu'.$$Then $\rho(G')> \rho(G)$ by Lemmas \ref{m9} and \ref{m11}. We will prove that $G'$ is a minimal matching covered bipartite graph in the sequel. In $G'$, remove the internal vertices of ears in $\mathcal {P}^{m}[u,v]$, the internal vertices of $P_{i}^{m}$ for $1\leq i\leq l$, the internal vertices of $P^{m-1}[v,v']$, and the internal vertices of $P^{m-1}[v,u]+uu'+P^{m-1}[u',v']$, denote the resulting graph by $G_{1}'$. Clearly, $G_{1}'\cong G^{*}_{1}$, where the vertices $u$ and $u'$ in $G^{*}_{1}$ correspond to $v$ and $v'$ in $G_{1}'$, respectively. Recall that $G^{*}=G^{*}_{1}+P^{m-1}$ is a minimal matching covered bipartite graph. Since $G$ is a minimal matching covered bipartite graph, the length of $P^{m-1}[v,v']$ is at least 3. So $G_{2}'=G_{1}'+P^{m-1}[v,v']$ is a minimal matching covered bipartite graph. Hence by Lemma \ref{m6}, we have $$G_{3}'=G_{2}'+(P^{m-1}[v,u]+uu'+P^{m-1}[v',u'])+\sum\limits_{i=1}\limits^{l}P^{m}_{i}$$ is a minimal matching covered bipartite graph. According to Lemma \ref{m12}, the $m$-grade parallel ears in $\mathcal{P}^{m}[u,v]$ are pairwise compatible. By Lemmas \ref{m6} and \ref{mx}, $G'=G_{3}'+\mathcal {P}^{m}[u,v]$ is a minimal matching covered bipartite graph. The proof is complete.                     
\end{proof} 

\begin{lemma}\label{z}
If there is exactly one family of $m$-grade parallel ears on $P^{m-1}$ such that one end is an internal vertex of $P^{m-1}$ and the other end is not on $P^{m-1}$, then there exists a minimal matching covered bipartite graph $G'$ satisfying $\rho(G')> \rho(G)$.
\end{lemma} 
\begin{proof}
Let $P^{m}_{1},P^{m}_{2},\ldots,P^{m}_{l}$ be the members of the family of $m$-grade parallel ears on $P^{m-1}$ with ends $v$ and $v'$, where $v$ is on $P^{m-1}$ and $v'$ is not on $P^{m-1}$. Recall that $u$ and $u'$ are the ends of $P^{m-1}$. Assume, without loss of generality, that $u$ and $v$ are in one part of $G$, then $u'$ and $v'$ are in other part. Since $G$ is a minimal matching covered bipartite graph, $v$ and $u'$ are nonadjacent by Lemma \ref{mg}. Denote the neighbor of $u'$ on $P^{m-1}$ by $u''$, and the neighbors of $v'$ on $P^{m}_{1},P^{m}_{2},\ldots,P^{m}_{l}$ by $v_{1}',\ldots,v_{l}'$, respectively. We consider the following two cases according to the relative values of $x_{u'}$ and $x_{v'}$.

\indent \textbf{Case\,1.} $x_{u'}\geq x_{v'}$. Let $G'=G-\sum\limits_{i=1}\limits^{l}v'v_{i}'+\sum\limits_{i=1}\limits^{l}u'v_{i}'$. Then $\rho(G')> \rho(G)$ by Lemma \ref{m11}. Now we prove that $G'$ is a minimal matching covered bipartite graph. The ear decomposition of $G'$ is obtained from the ear decomposition of $G$ by adjusting $P^{m-1}+\sum\limits_{i=1}\limits^{l}P^{m}_{i}$ to $P^{m-1}+\sum\limits_{i=1}\limits^{l}(P^{m}_{i}[v,v_{i}']+v_{i}'u')$. In the ear decomposition of $G'$, adjust $P^{m-1}+\mathcal {P}^{m}[u,u']+\sum\limits_{i=1}\limits^{l}(P^{m}_{i}[v,v_{i}']+v_{i}'u')$ to $(P^{m-1}[u,v]+P^{m}_{1}[v,v_{1}']+v_{1}'u')+\sum\limits_{i=2}\limits^{l}(P^{m}_{i}[v,v_{i}']+v_{i}'u')+P^{m-1}[v,u']+\mathcal {P}^{m}[u,u']$. Recall that $G^{*}$ is obtained from $G$ by removing all internal vertices of the ear $P^{m}_{i}$ for $1\leq i\leq l$ and all internal vertices of ears in $\mathcal{P}^{m}[u,u']$ and $G^{*}$ is a minimal matching covered bipartite graph. In $G'$, remove all internal vertices of the $m$-grade ears in $\mathcal{P}^{m}[u,u']$, all internal vertices of $P^{m-1}[v,u']$, and the internal vertices of $P^{m}_{i}[v,v_{i}']+v_{i}'u'$ for $2\leq i\leq l$, denote the resulting graph by $G_{1}'$. Note that $G_{1}'$ can be obtained from $G^{*}$ by replacing $P^{m-1}$ with $P^{m-1}[u,v]+P^{m}[v,v_{1}']+v_{1}'u'$. Then $G_{1}'$ is a minimal matching covered bipartite graph. By Lemma \ref{m12}, the ears in $\mathcal {P}^{m}[u,v]$ and $P^{m-1}[v,u']$ are compatible. According to Lemmas \ref{m6} and \ref{mx}, $G_{2}'=G_{1}'+P^{m-1}[v,u']+\mathcal {P}^{m}[u,v]$ is a minimal matching covered bipartite graph. Hence by Lemma \ref{m6}, $G_{3}'=G_{2}'+\sum\limits_{i=2}\limits^{l}(P^{m}_{i}[v,v_{i}']+v_{i}'u')$~is a minimal matching covered bipartite graph. Since the $m$-grade parallel ears on $P^{m-1}[v,u']$ are pairwise compatible, $G'=G_{3}'+\mathcal {P}^{m}[v,u']$ is a minimal matching covered bipartite graph.                              

\indent\textbf{Case\,2.}\,$x_{u'}< x_{v'}$. If there are ears with $u'$ as an end in $\mathcal {P}^{m}[v,u']$, then these ears have the same ends by Lemma \ref{m12}. Let $\mathcal{U'}$ be the set of neighbors of $u'$ on these ears. It is possible that $\mathcal{U'}=\emptyset$. Let $$G'=G-\sum\limits_{w\in \mathcal{U'}}u'w+\sum\limits_{w\in \mathcal{U'}}v'w-u'u''+v'u''.$$ Then by Lemma \ref{m11}, $\rho(G')> \rho(G)$. Next, we prove that $G'$ is a minimal matching covered bipartite graph. Adjust $P^{m-1}+\sum\limits_{i=1}\limits^{l}P^{m}_{i}$ in the ear decomposition of $G$ to $(P^{m-1}[u,v]+P^{m}_{1})+\sum\limits_{i=2}\limits^{l}P^{m}_{i}+P^{m-1}[v,u']$. Then in the new ear decomposition of $G$, $P^{m-1}[v,u']$ is an ear. Remove the internal vertices of ears in $\mathcal {P}^{m}[v,u']$ and the internal vertices of $P^{m-1}[v,u']$ from $G$, and denote the resulting graph by $G''$. By Lemma \ref{mg}, $G''$ is a minimal matching covered bipartite graph. Since the ears in $\mathcal{P}^{m}[u,v]$ are compatible with the ear $P^{m-1}[v,u'']+u''v'$, by Lemma \ref{mx}, $G_{1}'=G''+(P^{m-1}[v,u'']+u''v')$ is a minimal matching covered bipartite graph. By Lemma \ref{m12}, the $m$-grade parallel ears on $P^{m-1}[v,u']$ are compatible. Combining Lemma \ref{m6} with Lemma \ref{mx}, we have $G'=G_{1}'+\mathcal {P}^{m}[v,u']$ is a minimal matching covered bipartite graph. The proof is complete.                        
\end{proof}

\begin{lemma}\label{n}
If there are exactly two families of $m$-grade parallel ears on $P^{m-1}$ such that each has an end being an internal vertex of $P^{m-1}$ and the other is not on $P^{m-1}$, then there exists a minimal matching covered bipartite graph $G'$ satisfying $\rho(G')> \rho(G)$.    
\end{lemma}
\begin{proof}
Suppose there are two families of $m$-grade parallel ears on $P^{m-1}$, say,
$$\mathcal{P}^{m}_{1}=\{P^{m}_{11},P^{m}_{12},\ldots,P^{m}_{1l}\},~\mathcal{P}^{m}_{2}=\{P^{m}_{21},P^{m}_{22},\ldots,P^{m}_{2l'}\}, where~ l\geq1,~ l'\geq1.$$
Let the ends of the ears in $\mathcal{P}^{m}_{1}$ be $w$ and $w'$, and the ends of the ears in $\mathcal{P}^{m}_{2}$ be $v$ and $v'$, where $w$ and $v$ are on the ear $P^{m-1}$, $w'$ and $v'$ are not on the ear $P^{m-1}$. Then for any other $m$-grade ears on $P^{m-1}$ (if they exist), both ends are on $P^{m-1}$. Assume, without loss of generality, that the distance from $w$ to $u$ is less than the distance from $v$ to $u$. Note that $G$ is bipartite. The structures of $P^{m-1}$, $\mathcal{P}^{m}_{1}$\,and\,$\mathcal{P}^{m}_{2}$ can be divided into four different cases according to the parts of $G$ that $u, u', v, v', w, w'$ belong to, see Figure 3. Since $P^{m-1}$ is $(m-1)$-grade, at least one of the ends $u$ and $u'$ is an inner vertex of an $(m-2)$-grade ear, say $u$.                
\begin{figure}[htbp]
  \centering
  \includegraphics[width=1\textwidth]{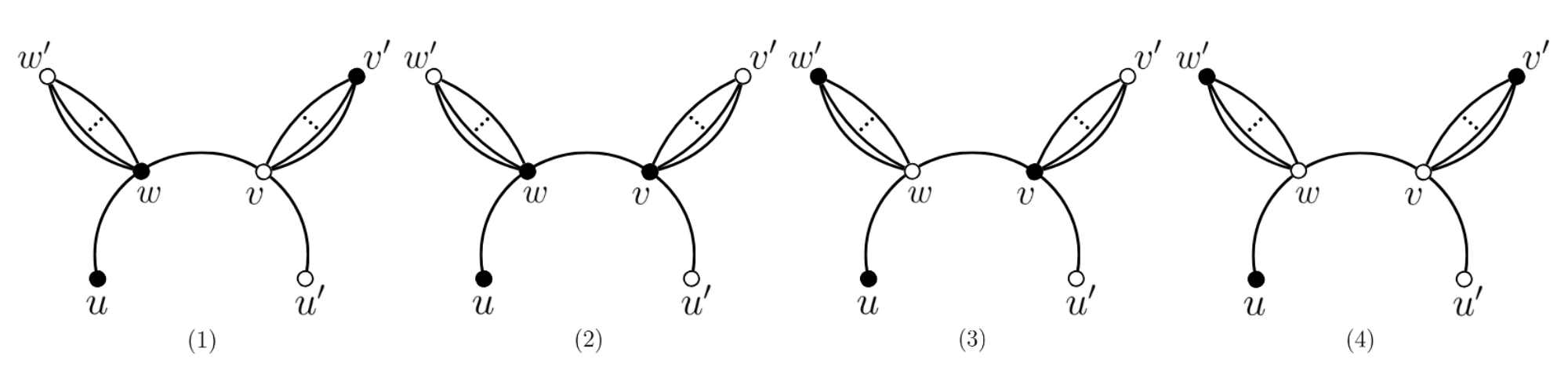}\\
  \caption{the local structures of $P^{m-1}$, $\mathcal{P}^{m}_{1}$\,and\,$\mathcal{P}^{m}_{2}$\,}
\end{figure}

If the structures of $P^{m-1}$, $\mathcal{P}^{m}_{1}$\,and\,$\mathcal{P}^{m}_{2}$ are as shown in Figure 3(1) and 3(2), then $u$~and~$w'$ belong to different parts of $G$ as well as $w$ and $u'$. Adjust $P^{m-1}+\mathcal{P}^{m}_{1}+\mathcal{P}^{m}_{2}$ in the original ear decomposition of $G$ to $(P^{m-1}[u,w]+P^{m}_{11})+\mathcal{P}^{m}_{1}\setminus \{P^{m}_{11}\}+P^{m-1}[w,u']+\mathcal{P}^{m}_{2}$. In the new ear decomposition of $G$, the grade of ear $P^ {m-1}[u,w]+P^{m}_{11}$ is at least $m-1$, and the grade of ear $P^{m-1}[w,u']$ is at least $m$. Therefore, the grade of ears in $\mathcal{P}^{m}_{2}$ is at least $m+1$. This contradicts the fact that the grade of $G$ is $m$. Therefore, the case shown in Figure 3(1) and 3(2) do not exist.            

If the structures of $P^{m-1}$, $\mathcal{P}^{m}_{1}$\,and\,$\mathcal{P}^{m}_{2}$ are as shown in Figure 3(4), adjust $P^{m-1}+\mathcal{P}^{m}_{1}+\mathcal{P}^{m}_{2}$ in the original ear decomposition of $G$ to $(P^{m-1}[u',v]+P^{m}_{21})+\mathcal{P}^{m}_{2}\setminus \{P^{m}_{21}\}+P^{m-1}[u,v]+\mathcal{P}^{m}_{1}$. In the new decomposition of $G$, among all the ears with ends on $P^{m-1}[u,v]$, there is exactly one family of parallel ears $\mathcal{P}^{m}_{1}$ satisfying one end is an internal vertex of $P^{m-1}[u,v]$ and the other end is not on $P^{m-1}[u,v]$. By Lemma \ref{z}, there exists a minimal matching covered bipartite graph $G'$ with greater spectral radius than that of $G$.          

If the structures of $P^{m-1}$, $\mathcal{P}^{m}_{1}$\,and\,$\mathcal{P}^{m}_{2}$ are as shown in Figure 3(3), then $u'$ and $v'$ are in one part of $G$, and $u$ and $w'$ are in the other part. Adjust $P^{m-1}+\mathcal{P}^{m}_{1}+\mathcal{P}^{m}_{2}$ in the ear decomposition of $G$ to $(P^{m-1}[u,v]+P^{m}_{21})+\mathcal{P}^{m}_{2}\setminus \{P^{m}_{21}\}+\mathcal{P}^{m}_{1}+P^{m-1}[v,u']$. In the new ear decomposition of $G$, $P^{m-1}[u,v]+P^{m}_{21}$ is an ear of $(m-1)$-grade, and $P^{m-1}[v,u']$ is an ear of $m$-grade. Since the grade of $G$ is $m$, there are no ears on $P^{m-1}[v,u']$. Since $G$ is a minimal matching covered bipartite graph, $u'$ and $v$ are nonadjacent. Denote the neighbor of $u'$ on the ear $P^{m-1}$ by $u''$, and the set of neighbors of $v'$ on $\mathcal{P}^{m}_{2}$ by $\{v_{1}',v_{2}',\ldots,v_{l'}'\}$. We will discuss the following two cases according to the relative values between $x_{u'}$~and~$x_{v'}$.              

If $x_{u'}\geq x_{v'}$, let $G'=G-\sum\limits_{i=1}\limits^{l'}v'v_{i}'+\sum\limits_{i=1}\limits^{l'}u'v_{i}'$. Then $\rho(G')> \rho(G)$ from Lemma \ref{m11}. Now we prove that $G'$ is a minimal matching covered bipartite graph. Remove the internal vertices of ears in $\mathcal{P}^{m}_{2}$ from $G$, and denote the resulting graph by $G''$. By Lemma \ref{mg}, $G''$ is a minimal matching covered bipartite graph. Note that $\cup_{i=1}^{l'}{P}^{m}_{2i}[v,v_{i}']+v_{i}'u'$ is compatible with $\mathcal{P}^{m}_{1}$, and the families of parallel ears of $\mathcal {P}^{m}[u,v]$. By Lemma \ref{mx}, $G'=G''+\sum\limits_{i=1}\limits^{l'}({P}^{m}_{2i}[v,v_{i}']+v_{i}'u')$ is a minimal matching covered bipartite graph.  

If $x_{u'}< x_{v'}$, let $G'=G-u''u'+u''v'$. Then $\rho(G')> \rho(G)$. Next, we prove that $G'$ is also a minimal matching covered bipartite graph. Recall that $P^{m-1}[v,u']$ is the last ear in the new ear decomposition of $G$ described above. Remove the internal vertices of $P^{m-1}[v,u']$ from $G$, and denote the resulting graph by $G''$. By Lemma \ref{mg}, $G''$ is a minimal matching covered bipartite graph. Since ${P}^{m-1}[v,u'']+u''v'$ is compatible with $\mathcal{P}^{m}_{1}$, and the families of parallel ears in $\mathcal {P}^{m}[u,v]$, by Lemma \ref{mx}, $G'=G''+{P}^{m-1}[v,u'']+u''v'$ is a minimal matching covered bipartite graph. The proof is complete. 
\end{proof} 
\vspace{0.2cm}                        

Now, we are in a position to present the proof of Theorem 1.1. 
\vspace{0.2cm}

\noindent\textbf{Proof of Theorem 1.1.}
Let $G$ be the graph with the greatest spectral radius among minimal matching covered bipartite graphs of order $n$. If $n=2$, then $G\cong K_{2}$ and $\rho(G)=1\leq 2$. If $n=4$, then $G\cong C_{4}$ and $\rho(G)=2$. If $n=6$, then the nice cycle of $G$ has no chords since $G$ is a minimal matching covered bipartite graph. Then one can check that $G\cong C_{6}$ and $\rho(G)=2$. We now consider the case when $n\geq 8$ in the sequence.

First, we prove that $G$ is not a cycle. By contradiction and suppose that $G$ is a cycle, and denote the vertices on the cycle by $v_{1}, v_{2}, \ldots, v_{n}$ in order. Let $G'=G-v_ {1}v_{2}+v_{1}v_{4}+v_{2}v_{n-1}$. Then $G'$ is the union of a cycle of length 6 and an ear $v_{4}v_{5}\cdots v_{n-1}$ of length at least 3. By Lemma \ref{m2}, $G'$ is a matching covered bipartite graph. It can be check that each cycle in $G'$ has no chords. Hence by Lemma \ref{m5}, $G'$ is a minimal matching covered bipartite graph. Let \textbf{x} be the Perron vector of $G$. By Lemma \ref{m10}, $x_ {v_{1}}=x_{v_{2}}=\cdots =x_{v_{n}}$. According to Lemmas \ref{m9} and \ref{m11}, $\rho(G')> \rho(G'-v_2v_{n-1})>\rho(G)$, which contradicts the choice of $G$. Therefore, $G$ is not a cycle.                     

By Lemma \ref{m2}, $G$ has bipartite ear decompositions. Suppose that the grade of $G$ is $m$. Then the ear decomposition of $G$ with grade $m$ can be expressed as $G=C+ \mathcal {P}^{1}+ \mathcal {P}^{2}+\cdots+ \mathcal{P}^{m}$, where $\mathcal {P}^{i}$ is the set of all $i$-grade ears of $G$, $1\leq i\leq m$. 

\indent\textbf{Claim~1.} $m=1$.    

Suppose that $m\geq 2$. Let $P^{m-1}$ be an $(m-1)$-grade ear of $G$ with ends $u$ and $u'$. Let $k$ be the number of families of $m$-grade parallel ears such that for each family, one end of ears in it is an internal vertex of $P^{m-1}$ and the other end is not on $P^{m-1}$. We will prove by induction on $k$ that there exists a minimal matching covered bipartite graph $G'$ satisfying $\rho(G')> \rho(G)$, that is a contradiction to the choice of $G$.      

If $k=0, 1, 2$, then there exists a minimal matching covered bipartite graph $G'$ satisfying $\rho(G')> \rho(G)$ by Lemmas \ref{a}, \ref{z} and \ref{n}. Hence we only need to consider the case when $k\geq3$. By Lemma \ref{m12}, these $k$ families of parallel ears are pairwise compatible. Denote the family of parallel ears whose ends are closest to $u$ by $\mathcal{P}^{m}_{1}$, and the family of parallel ears whose ends are closest to $u'$ by $\mathcal{P}^{m}_{2}$. Let$$\mathcal{P}^{m}_{1}=\{P^{m}_{11},P^{m}_{12},\ldots,P^{m}_{1l}\},~\mathcal{P}^{m}_{2}=\{P^{m}_{21},P^{m}_{22},\ldots,P^{m}_{2l'}\},~l\geq1,~l'\geq1.$$ 
Let the ends of ears in $\mathcal{P}^{m}_{1}$ be $w$ and $w'$, and those in $\mathcal{P}^{m}_{2}$ be $v$ and $v'$. Depending on the parts $u$, $u'$, $w$, $w'$, $v$, $v'$ are in, there are four cases of the local structures of $P^{m-1}$, ears in $\mathcal{P}^{m}_{1}$\,and\,$\mathcal{P}^{m}_{2}$\,in $G$, as shown in Figure 3.

If the local structures of $P^{m-1}$, ears in $\mathcal{P}^{m}_{1}$~and~$\mathcal{P}^{m}_{2}$ in $G$ are as shown in Figure 3(1). Similar to the argument in the proof of Lemma \ref{n}, we can adjust the ear decomposition of $G$ so that the grade of new ear decomposition is greater than $m$. This contradicts the fact that the grade of $G$ is $m$.  
            
If the local structures of $P^{m-1}$, ears in $\mathcal{P}^{m}_{1}$~and~$\mathcal{P}^{m}_{2}$ in $G$ are as shown in Figure 3(2) and Figure 3(4). By symmetry, it suffices to consider the case in Figure 3(2). We adjust the ear decomposition of $G$: replacing $P^ {m-1}+\mathcal{P}^{m}_{1}+\mathcal{P}^{m}_{2}$ with $P^{m-1}[u,w]+P^{m}_{11}+\mathcal{P}^{m}_{1}\setminus \{P^{m}_{11}\}+P^{m-1}[w,u']+\mathcal{P}^{m}_{2}$. In the new decomposition of $G$, there are $k-1$ families of parallel ears on $P^{m-1}[w,u']$ that have one end being an internal vertex of $P^{m-1}[w,u']$ and the other end not on $P^{m-1}[w,u']$, and from the choice of $G$, we know that the grade of $P^{m-1}[w,u']$ is at most $m-1$. According to the inductive hypothesis, there is a minimal matching covered bipartite graph $G'$ satisfying $\rho(G')> \rho(G)$.     

If the local structures of $P^{m-1}$, ears in $\mathcal{P}^{m}_{1}$~and~$\mathcal{P}^{m}_{2}$ in $G$ are as shown in Figure 3(3). Note that there are $k-2$ families of $m$-grade parallel ears on $P^{m-1}[w,v]$ that has one end being an internal vertex of $P^{m-1}[w,v]$ and the other end not on $P^{m-1}[w,v]$. Among these $k-2$ families of $m$-grade parallel ears, arbitrarily select one family and denote it by $\mathcal{P}^{m}_{3}$. Let $\mathcal{P}^{m}_{3}=\{P^{m}_{31},P^{m}_{32},\ldots,P^{m}_{3p}\}$, where $p\geq1$. Denote the ends of ears in $\mathcal{P}^{m}_{3}$ by $y$ and~$y'$. Assume, without loss of generality, that $y$ is an internal vertex of $P^{m-1}[w,v]$. If $u$ and $y$ are in different parts of $G$, then $u'$ and $y$ are in the same part. By symmetry, it suffices to consider the case where $u$ and $y$ are in the same part. Adjust $P^{m-1}+\mathcal{P}^{m}_{1}+\mathcal{P}^{m}_{2}+\mathcal{P}^{m}_{3}$ in the ear decomposition of $G$ to $(P^{m-1}[u,y]+P^{m}_{31})+\mathcal{P}^{m}_{3}\setminus \{P^{m}_{31}\}+\mathcal{P}^{m}_{1}+P^{m-1}[y,u']+\mathcal{P}^{m}_{2}$. Then there are fewer than $k$ families of parallel ears on $P^{m-1}[y,u']$ that has one end being an internal vertex of $P^{m-1}[y,u']$ and the other end not on $P^{m-1}[y,u']$. Since the grade of $G$ is $m$, we know that the grade of $P^{m-1}[y,u']$ is at most $m-1$. According to the inductive hypothesis, there is a minimal matching covered bipartite graph $G'$ satisfying $\rho(G')> \rho(G)$, a contradiction. The proof of Claim 1 is complete.                

By Claim 1, the ear decomposition of $G$ is $C+\mathcal{P}^{1}$. If there are two families of parallel ears $\mathcal{P}^{1}_{1}$~and~$\mathcal{P}^{1}_{2}$ with different ends in $\mathcal{P}^{1}$, by Lemma \ref{m12}, $\mathcal{P}^{1}_{1}$ and $\mathcal{P}^{1}_{2}$ are compatible. Adjusting the ear decomposition $C+\mathcal{P}^{1}$, we can see that the grade of $G$ is at least 2, a contradiction. Thus $\mathcal{P}^{1}$ is a family of parallel ears. We denote the ends of the ears in $\mathcal{P}^{1}$ by $u$ and $u'$. Since $G$ is a minimal matching covered bipartite graph, every nice cycle of $G$ has no chords. Hence $u$ and $u'$ are nonadjacent in $G$.   

\indent\textbf{Claim~2.} The length of every cycle in $G$ is 6.

Suppose there exists a cycle $C'$ in $G$ whose length is greater than 6. Let $C'+P_{1}+P_{2}+\cdots+P_{k}$ be an ear decomposition of $G$ starting from $C'$. Since the length of $C'$ is greater than 6, there exists a $(u, u')$-path of length greater than 3 on cycle $C'$, denoted by $P_{uu'}$. Denote the neighbors of $u$ and $u'$ on path $P_{uu'}$ by $v$ and $v'$, respectively, and the neighbors of $u$ and $u'$ on ear $P_{i}$ as $u_{i}$ and $u_{i}'$, respectively, where $1\leq i\leq k$. Except for $u$ and $u'$, denote the neighbors of $v$ and $v'$ on path $P_{uu'}$ by $w$ and $w'$, respectively. By Lemma \ref{m10}, we have $x_{u}=x_{u'}$, $x_{v}=x_{v'}$.

If $x_{v}>x_{u}$, then $x_{v'}>x_{u'}$. Let$$G'=G-\sum\limits_{i=1}\limits^{k}uu_{i}+\sum\limits_{i=1}\limits^{k}vu_{i}-\sum\limits_{i=1}\limits^{k}u'u_{i}'
+\sum\limits_{i=1}\limits^{k}v'u_{i}'.$$One can check that every nice cycle of $G'$ has no chords, thus $G'$ is a minimal matching covered bipartite graph by Lemma \ref{m5}. It follows from Lemma \ref{m11} that $\rho(G')> \rho(G)$, a contradiction to the choice of $G$.               

If $x_{v}\leq x_{u}$, then $x_{v'}\leq x_{u'}$. Let $G'=G-wv+wu-w'v'+w'u'+vv'$. It can be checked that every nice cycle of $G'$ has no chords. Then $G'$ is a minimal matching covered bipartite graph by Lemma \ref{m5}. By Lemmas \ref{m9} and \ref{m11}, we have $\rho(G')> \rho(G)$, a contradiction. Therefore, the length of cycles in $G$ is at most 6. Since $G$ is a minimal matching covered bipartite graph, the length of cycles in $G$ is greater than 4 by Lemma \ref{m4}. Thus the length of every cycle in $G$ is 6. Claim 2 is proved.       

For the ear decomposition $C+\mathcal{P}^{1}$ of $G$, we have $C$ is a 6-cycle and the ears in $\mathcal{P}^{1}$ all have length 3 by Claim~1. It follows that $G\cong P^{*}_{3}$. Hence $\rho(G)=\frac{1+\sqrt{2n-3}}{2}$ by Lemma \ref{mb}. The proof of Theorem 1.1 is complete.     \qed

\section{Minimal factor-critical graphs}
In this section, we utilize the ear decomposition of factor-critical graphs to determine the upper bound for the spectral radius of minimal factor-critical graphs and characterize extremal graphs. First, we introduce some definitions and lemmas, which will lay the foundation for subsequent proofs.
\begin{definition}
If $B$ is a maximal connected subgraph of a graph $G$ and has no cut vertices, then $B$ is called a block of $G$. In particular, $B$ is called an end block of $G$ if $B$ contains only one cut vertex of $G$.
\end{definition}

A graph is said to have an odd-ear decomposition if it can be constructed by starting from an odd cycle and successively adding ears (paths of odd length). Referencing exercises 16.3.11 and 16.3.12 in [3], we can obtain a characterization of factor-critical graphs without cut vertices in terms of odd-ear decompositions.

\begin{lemma}[\cite{ref50}]\label{m19}
A graph is a factor-critical graph without cut vertices if and only if it has an odd-ear decomposition.
\end{lemma}

For factor-critical graphs with cut vertices, the following characterization is provided.

\begin{lemma}[\cite{ref17}]\label{m18}
A graph is factor-critical if and only if it is connected and each of its block is factor-critical.
\end{lemma}

There is also a characterization for minimal factor-critical graphs similar to Lemma \ref{m18}. 

\begin{lemma}[\cite{ref17}]\label{m20}
A graph is a minimal factor-critical graph if and only if each of its block is a minimal factor-critical graph.
\end{lemma}

\begin{lemma}[\cite{ref17}]\label{m-ear}
In every ear decomposition of a minimal factor-critical graph, all ears are nontrivial.
\end{lemma}

Now we are in a position to present the proof of Theorem 1.2. 
\vspace{0.2cm}

\noindent\textbf{Proof of Theorem 1.2.} Let $G$ be a minimal factor-critical graph of order $n$ with the greatest spectral radius. Let \textbf{x} be the Perron vector of $G$. We start by proving two claims.

\indent\noindent\textbf{Claim 1.} Every block of $G$ is an odd cycle.

Suppose that a block $B$ of $G$ is not an odd cycle. By Lemmas \ref{m19} and \ref{m-ear}, $B$ has an odd-ear decomposition~$P_{0}+P_{1}+\cdots+P_{r}$, where $P_{0}$ is an odd cycle, and for each $i$, $0\leq i\leq r-1$, $P_{i+1}$ is an ear of $P_{0}+\cdots+P_{i}$ of length at least 3. Let the two ends of $P_{r}$ be $u$ and $v$, and the neighbors of $u$ and $v$ on $P_{r}$ be $u'$ and $v'$, respectively. Assume, without loss of generality, that $x_{u}\geq x_{v}$. Let $B'=B-vv'+uv'$. Then $P_{0}+P_{1}+\cdots+P_{r-1}+(P_{r}[u,v']+uv')$ is an ear decomposition of $B'$. Write $B''=P_{0}+P_{1}+\cdots+P_{r-1}$. Clearly, $B''$ and the odd cycle $P_{r}[u,v']+uv'$ are two blocks of $B'$. By Lemma \ref{m20}, both $B'$ and $B''$ are minimal factor-critical graphs. Note that $G'=G-vv'+uv'$. Since the odd cycle $P_{r}[u,v']+uv'$ is a minimal factor-critical graph, by Lemma \ref{m20}, $G'$ is a minimal factor-critical graph. By Lemma \ref{m11}, $\rho(G')>\rho(G)$, which contradicts the choice of $G$. Claim 1 is proved. 

\indent\noindent\textbf{Claim 2.} Every block of $G$ is $C_{3}$.                         

Let $C$ be a block of $G$. By Claim 1, $C$ is an odd cycle. Assume that the length of $C$ is $k$ and $k\geq 5$. Label the vertices on the cycle $C$ as $v_{1}, v_{2}, \ldots, v_{k}$ in order. Similar to the analysis in the proof of Theorem 1.1, we know $G$ is not a cycle, then there are cut vertices in $G$. Let $v_1$ be a cut vertex of $G$. Compare the values of $x_{v_{2}}$ and $x_{v_{4}}$. If $x_{v_{2}}<x_{v_{4}}$, let $G'=G-v_{1}v_{2}+v_{1}v_{4}+v_{2}v_{4}$. If $x_{v_{2}}\geq x_{v_{4}}$, let $G'=G-v_{4}v_{5}+v_{5}v_{2}+v_{4}v_{2}$. Note that every block of $G'$ is an odd cycle, which is factor-critical. By Lemma \ref{m18}, $G'$ is a factor-critical graph. Since odd cycles are minimal factor-critical graphs, by Lemma \ref{m20}, $G'$ is a minimal factor-critical graph. By Lemmas \ref{m11} and \ref{m9}, $\rho(G')>\rho(G)$, a contradiction to the choice of $G$. Claim 2 is proved.         

From Claim 2, we know that every block of $G$ is $C_3$. If $G$ has only one cut vertex, then the conclusion holds. Suppose that $G$ has at least two cut vertices. Let $B_1$ and $B_2$ be two end blocks of $G$, and $w_1$ and $w_2$ be the cut vertices of $G$ in $B_1$ and $B_2$, respectively. Assume that $w_1$ and $w_2$ is a pair of cut vertices with the maximum distance in $G$. Assume, without loss of generality, that $x_{w_2}\geq x_{w_1}$. Let $w$ and $w'$ be the neighbors of $w_{1}$ in $B_1$. Let $G'=G-w_1w-w_1w'+w_2w+w_2w'$. By Lemma \ref{m20}, $G'$ is a minimal factor-critical graph. By Lemma \ref{m11}, $\rho(G')>\rho(G)$, a contradiction. Therefore, $G$ has only one cut vertex, and so $G\cong K_{1}\vee\frac{n-1}{2}K_{2}$. Theorem 1.2 is proved.\qed

\section*{Acknowledgement}

The authors are supported by National Natural Science Foundation of China (12571381, 12201574, 12371361) and Natural Science Foundation of Henan Province (252300420303).

\end{document}